\documentclass[%reqno,
12pt]{amsart}
\usepackage{amscd}
\usepackage{amssymb}
\usepackage{a4wide}
\usepackage{amstext}
\usepackage[english]{babel}
\usepackage[latin1]{inputenc}
\usepackage{times}
\usepackage[T1]{fontenc}
\usepackage{mathrsfs}
\usepackage[final]{hyperref}
\hypersetup{unicode= false, colorlinks=true, linkcolor=blue,
anchorcolor=blue, citecolor=green, filecolor=red, menucolor=blue,
pagecolor=blue, urlcolor=blue} 
\newcommand{\Hg}{\mathscr{H}}

\newcommand{\CC}{\mathbb{C}}

\newtheorem{theorem}{Theorem}

\theoremstyle{definition}

\title[Unitary discrete Hilbert transforms]
{Unitary discrete Hilbert transforms}
\author [Yurii Belov]{Yurii Belov }
\address{Department of Mathematical Sciences\\
Norwegian University of Science and Technology (NTNU)\\
 NO- 7491 Trondheim, Norway}
 \email{j\_b\_juri\_belov@mail.ru}
\author [Tesfa Y. Mengestie]{Tesfa Y. Mengestie }
\address{Department of Mathematical Sciences\\
Norwegian University of Science and Technology (NTNU)\\
 NO- 7491 Trondheim, Norway}
\email{mengesti@math.ntnu.no}
\author[Kristian Seip]{Kristian Seip}
\address{Department of Mathematical Sciences\\
Norwegian University of Science and Technology (NTNU)\\
 NO- 7491 Trondheim, Norway}
\email{seip@math.ntnu.no}
\thanks{The authors are supported by the Research Council of
Norway grant 185359/V30.} \subjclass[2000]{30H10, 46E22}
\begin{document}
\begin{abstract}
Weighted discrete Hilbert transforms $(a_n)_n \mapsto \big(\sum_n
a_n v_n/(\lambda_j-\gamma_n)\big)_j$ from $\ell^2_v$ to $\ell^2_w$
are considered, where $\Gamma=(\gamma_n)$ and $\Lambda=(\lambda_j)$
are disjoint sequences of points in the complex plane and $v=(v_n)$
and $w=(w_j)$ are positive weight sequences. It is shown that if
such a Hilbert transform is unitary, then $\Gamma\cup\Lambda$ is a
subset of a circle or a straight line, and a description of all
unitary discrete Hilbert transforms is then given. A
characterization of the orthogonal bases of reproducing kernels
introduced by L. de Branges and D. Clark is implicit in these
results: If a Hilbert space of complex-valued functions defined on a
subset of $\CC$ satisfies a few basic axioms and has more than one
orthogonal basis of reproducing kernels, then these bases are all of
Clark's type.
\end{abstract}
\maketitle

\section{Introduction}

If we are given two finite or infinite sequences of distinct points
$\Gamma=(\gamma_n)$ and $\Lambda=(\lambda_j)$ in $\CC$ and a
sequence of positive numbers $v=(v_n)$, we may define the discrete
Hilbert transform by
\begin{equation} (a_n)_n \mapsto \left(\sum_n \frac{a_n
v_n}{\lambda_j-\gamma_n} \right)_j. \label{discreteH}
\end{equation}
To make sense of this, we assume that $\Gamma$ and $\Lambda$, viewed
as subsets of $\CC$, are disjoint. We also assume that $\Lambda$ is
a subset of the set
\[ (\Gamma, v)^{*}=\left\{ z\in \CC:\
\sum_n\frac{v_n}{|z-\gamma_n|^2}<\infty\right\} \] because we wish
to define the discrete Hilbert transform for sequences $(a_n)$ in
\[ \ell^2_v=\{ (a_n):\ \sum_n |a_n|^2 v_n<\infty\}.\]
If we now associate another weight sequence $w=(w_j)$ with
$\Lambda$, we may ask: When is the discrete Hilbert transform
$H_v(\Gamma,\Lambda)$ given by \eqref{discreteH} a unitary
transformation from $\ell^2_v$ to $\ell^2_w$?

The present note answers this question and shows how the solution
translates into a general statement about orthogonal bases of
reproducing kernels. Making a few minimal assumptions on the
underlying Hilbert space, we arrive at the following conclusion:
There are no other orthogonal bases of reproducing kernels than
those introduced and studied by L. de Branges \cite{DB} and D. Clark
\cite{CL}.

\section{Localization of the sequences $\Gamma$ and $\Lambda$}

Our starting point is the following localization result.

\begin{theorem}
If the discrete Hilbert transform
\[H_v(\Gamma,\Lambda):\ \ell^2_v\to \ell^2_w\]
is unitary, then $\Gamma\cup\Lambda$ is a subset of a circle or a
straight line in $\CC$.
\end{theorem}

\begin{proof}
In what follows, we let $e^{(n)}$ denote the vectors in the standard
orthonormal basis for $\ell^2_v$. Thus $e^{(n)}$ is the sequence for
which the $n$-th entry is $v_n^{-1/2}$ and all the other entries are
$0$.

We fix an index $m$ and observe that since $\Gamma$ is a subset of
$(\Lambda,w)^*$, the function
\[ G(z)=
(z-\gamma_m)\sum_{j}
\frac{w_j}{(\overline{\lambda_j}-\overline{\gamma_m}) (\lambda_j-z)}
\] is well defined for $z$ in $\Gamma$. In fact, since $H_v(\Gamma,\Lambda)$ is assumed
to be a unitary transformation, the basis vectors $e^{(n)}$ map into
an orthonormal system in $\ell_w^2$, and therefore $G$ vanishes on
$\Gamma$. Thus we may write
\[ G(z)=G(z)-G(\gamma_n)=(z-\gamma_n)\sum_{j}
\frac{w_j(\lambda_j-\gamma_m)}{(\overline{\lambda_j}-\overline{\gamma_m})
(\lambda_j-\gamma_n)(\lambda_j-z)},\] where on the right-hand side
we have just subtracted the respective series that define $G(z)$ and
$G(\gamma_n)$. It follows that
\[ \frac{G(z)}{z-\gamma_n}=\sum_{j}
\frac{w_j(\lambda_j-\gamma_m)}{(\overline{\lambda_j}-\overline{\gamma_m})
(\lambda_j-\gamma_n)(\lambda_j-z)}, \] and this function vanishes
for $z$ in $\Gamma\setminus\{\gamma_n\}$. Since
$H_v(\Gamma,\Lambda)$ is
 assumed to be unitary, the vectors $H_v(\Gamma,\Lambda) e^{(n)}$
constitute an orthonormal basis for $\ell^2_w$, and therefore the
sequence
\[
\left(\frac{\lambda_j-\gamma_m}{\overline{\lambda_j}-\overline{\gamma_m}}\cdot
\frac{1}{\lambda_j-\gamma_n}\right)_j
\]
is a multiple of the sequence
$(1/(\overline{\lambda_j}-\overline{\gamma_n}))_j$. Thus the complex
numbers $(\lambda_j-\gamma_m)^2/(\lambda_j-\gamma_n)^2$ have the
same argument for all $j$, and so
\[
\left(\frac{(\lambda_j-\gamma_m)(\lambda_l-\gamma_n)}
{(\lambda_j-\gamma_n)(\lambda_l-\gamma_m)}\right)^2>0
\]
for $j\neq l$ and $m\neq n$. In other words, the cross ratio of the
four complex numbers $\lambda_j$, $\lambda_l$, $\gamma_n$,
$\gamma_m$ is real. This can only happen if the points lie on the
same circle or straight line.

After having applied this argument to four arbitrary points, say
$\lambda_1$, $\lambda_2$, $\gamma_1$, $\gamma_2$, we see that in
fact every point from $\Gamma\cup\Lambda$ lie on the circle or
straight line determined by the four initial points, because we may
apply the same argument to any given point in $\Gamma\cup\Lambda$
along with three of the points $\lambda_1$, $\lambda_2$, $\gamma_1$,
$\gamma_2$.
%(Indeed, we may use the fact that the cross ratio is invariant
%under M\"{o}bius transformations and therefore assume that
%$\lambda_j=\infty$, $\lambda_l=0$, and $\gamma_n=1$, which imply
%that $\gamma_m$ is a real number.)
\end{proof}

\section{The unitary transformations associated with $\Gamma$ and $v$}

For a given sequence $\Gamma$ being a subset of a circle or straight
line and an associated weight sequence $v$, we wish to describe
those pairs $\Lambda$ and $w$ such that $H_v(\Gamma,\Lambda): \
\ell^2_v\to\ell^2_w$ is a unitary transformation. To begin with, we
require the admissibility condition
\begin{equation}
\sum_n \frac{v_n}{1+|\gamma_n|^2}<\infty, \label{admissible}
\end{equation}
which is now a necessary and sufficient condition for
$(\Gamma,v)^{*}$ to be nonempty; we will say that $(\Gamma,v)$ is an
admissible pair whenever \eqref{admissible} holds.

We will assume that $\Gamma$ is a subset of the real line. The case
when $\Gamma$ is a subset of a circle is completely analogous, as
will be briefly commented on at the end of this section. We set
\begin{equation} \label{phi}\varphi(z)=\sum_n
v_n\left(\frac{1}{\gamma_n-z}-\frac{\gamma_n}{1+\gamma_n^2}\right)
\end{equation}
and observe that $\varphi$ is well-defined on $(\Gamma,v)^*$ because
the series in \eqref{phi} converges absolutely for $z$ in
$(\Gamma,v)^*$. We also note that $\varphi$ is a Herglotz function
in the upper half-plane (\cite{CY}, Chapter 9 ). In particular, it
is a holomorphic function whose imaginary part is positive. A
general Herglotz function $\psi$ in the upper half-plane can be
written as \[ \psi(z)=b+ c z + \int_{-\infty}^\infty
\left(\frac{1}{t-z}-\frac{t}{1+t^2}\right)d\mu(t), \] where $b$ is a
real constant, $c$ a positive constant, and $\mu$ a nonnegative
measure on the real line such that
\[ \int_{-\infty}^\infty \frac{d\mu(t)}{1+t^2}<\infty. \]   We will
say that $\psi$ is a purely atomic Herglotz function if $c=0$ and
$\mu$ is a purely atomic measure; our function $\varphi$ is thus an
example of a purely atomic Herglotz function.

For every real number $\alpha$, we set
\[ \Lambda(\alpha)=\{\lambda\in (\Gamma,v)^*: \ \varphi(\lambda)=\alpha\}.
 \]
We observe that \begin{equation} \label{ortho} \sum_n \frac{v_n
(z-w)}{(w-\gamma_n)(z-\gamma_n)}=
\varphi(z)-\varphi(w),\end{equation} which implies that the
sequences $(1/(\lambda-\gamma_n))_n$ with $\lambda$ in
$\Lambda(\alpha)$ constitute an orthogonal set in $\ell^2_v$. This
means that $\Lambda(\alpha)$ is at most a countable set, so that we
may associate with $\Lambda(\alpha)$ a weight sequence
$w(\alpha)=(w_j)$, where
\begin{equation}
w_j=\left(\sum_{n}\frac{v_n}{(\lambda_j-\gamma_n)^2}\right)^{-1}
\label{weight}
\end{equation}
for $\lambda_j$ in $\Lambda(\alpha)$. It is implicit in our
arguments that if $H_v(\Gamma,\Lambda): \ell^2_v \to \ell_w^2$ is a
unitary transformation, then $\Lambda=\Lambda(\alpha)$ and
$w=w(\alpha)$ for some real number $\alpha$.

We will now prove the following theorem.

\begin{theorem}
Let $(\Gamma, v)$ be an admissible pair with $\Gamma$ a subset of
the real line, and let $\alpha$ be a real number. Then the discrete
Hilbert transform
\[H_v(\Gamma,\Lambda(\alpha)):\ell^2_v \to \ell^2_{w(\alpha)}\]
is unitary if and only if $(\alpha-\varphi(z))^{-1}$ is a purely
atomic Herglotz function. \label{theorem2}
\end{theorem}

\begin{proof}
In this proof, we will again use the standard orthonormal basis
vectors $e^{(n)}$ in $\ell^2_v$; we will denote the corresponding
basis vectors in $\ell^2_{w(\alpha)}$ by $f^{(j)}$. We will use the
notation $\|\cdot \|_v$ and $\| \cdot\|_w$ for the respective norms
in $\ell^2_v$ and $\ell^2_w$.

It is clear that the adjoint transformation to
$H_v(\Gamma,\Lambda(\alpha))$ is again a discrete Hilbert transform.
In fact, since $\Gamma$ and $\Lambda(\alpha)$ are sequences of real
numbers, we have
$H_v^*(\Gamma,\Lambda(\alpha))=-H_{w(\alpha)}(\Lambda,\Gamma)$,
where
\[H_{w(\alpha)}(\Lambda(\alpha),\Gamma):\ell^2_{w(\alpha)}\to \ell^2_v.\]
Therefore, $H_v(\Gamma,\Lambda(\alpha))$ is unitary if and only if
both $H_v(\Gamma,\Lambda(\alpha))$ and
$H_{w(\alpha)}(\Lambda(\alpha), \Gamma)$ are isometric. Hence it
suffices to check whether $(H_v(\Gamma,\Lambda(\alpha)) e^{(n)})$
and $(H_{w(\alpha)}(\Lambda(\alpha),\Gamma)f^{(j)})$ are orthonormal
sequences in respectively $\ell^2_{w(\alpha)}$ and $\ell^2_v$.

The orthogonality of the vectors
$H_{w(\alpha)}(\Lambda(\alpha),\Gamma)f^{(j)}$ in $\ell^2_v$ has
already been verified (see \eqref{ortho}); it is just a consequence
of the definition of $\Lambda(\alpha)$. Likewise, we have
automatically
\[ \| H_{w(\alpha)}(\Lambda(\alpha),\Gamma)f^{(j)}\|_{v}^2
= \sum_n \frac{w_j v_n}{|\gamma_n-\lambda_j|^2} =1.\] So our task is
to show that $(H_v(\Gamma,\Lambda) e^{(n)})$ is an orthonormal
sequence in $\ell^2(\Lambda(\alpha),w(\alpha))$ if and only if
$(\alpha-\varphi(z))^{-1}$ is a purely atomic Herglotz function.

We first assume that $(\alpha-\varphi(z))^{-1}$ is indeed a purely
atomic Herglotz function. It suffices to show that there is a real
constant $b$ such that
\begin{equation}
\label{toprove} \frac{1}{\alpha-\varphi(z)}=b+ \sum_jw_j
\left(\frac{1}{\lambda_j-z}-\frac{\lambda_j}{1+\lambda_j^2}\right),
\end{equation}
where $\lambda_j$ are the points in $\Lambda(\alpha)$ and $w_j$ are
as in \eqref{weight}. Indeed, by symmetry, it will then follow that
the numbers $\gamma_n$ are solutions to the equation
\[\sum_j w_j \left(\frac{1}{\lambda_j-z}-\frac{\lambda_j}{1+\lambda_j^2}\right)=-b,\]
so that the arguments already employed for the vectors $
H_{w(\alpha)}(\Lambda(\alpha),\Gamma)f^{(j)}$ apply similarly to the
vectors $H_v(\Gamma,\Lambda(\alpha))e^{(n)}$.

We start from the representation \eqref{toprove}, with no a priori
assumption on the points $\lambda_j$ and the nonnegative numbers
$w_j$ except the admissibility condition
\[ \sum_j \frac{w_j}{1+\lambda_j^2} <\infty; \]
our goal is to prove that the $\lambda_j$ are in $\Lambda(\alpha)$
and that the $w_j$ are given by \eqref{weight}. We first observe
that if we set $z=\lambda_j+iy$, then we get, by restricting to
imaginary parts,
\[ \frac{w_j}{y}\le \left(\sum_{n}
\frac{yv_n}{(\lambda_j-\gamma_n)^2+y^2}\right)^{-1}, \] whence \[
\sum_{n} \frac{v_n}{(\lambda_j-\gamma_n)^2}\le w_j^{-1}. \] In other
words, the points $\lambda_j$ belong to $(\Gamma, v)^{*}$. We now
multiply each side of \eqref{toprove} by $z-\lambda_j$ and take the
limit when $z=\lambda_j+iy$ and $y\to 0^+$; since $\lambda_j$ is in
$(\Gamma,v)^{*}$ and $\varphi(\lambda_j)=\alpha$, this gives
\eqref{weight}.

Suppose, on the other hand, that $(\alpha-\varphi(z))^{-1}$ is not a
purely atomic Herglotz function and that the vectors
$H_v(\Gamma,\Lambda(\alpha))e^{(n)}$ constitute an orthonormal
system in $\ell^2_{w(\alpha)}$. We will show that this leads to a
contradiction. To begin with, our assumption on
$(\alpha-\varphi(z))^{-1}$ implies that
\begin{equation} \frac{1}{\alpha-\varphi(z)}=b+ \sum_jw_j
\left(\frac{1}{\lambda_j-z}-\frac{\lambda_j}{1+\lambda_j^2}\right) +
cz + \int_{-\infty}^\infty
\left(\frac{1}{t-z}-\frac{t}{1+t^2}\right)d\mu(t), \label{nonatom}
\end{equation} with $\mu$ a spectral measure such that
$\mu(\{\lambda_j\})=0$ for every $j$ and not both $c=0$ and $\mu=0$;
the fact that the $w_j$ are given by \eqref{weight} can be proved as
in the first part of the proof.

We now argue in the same way as above, reversing the roles of
$\Gamma$ and $\Lambda(\alpha)$. This means that we first show, by
again restricting to imaginary parts, that
\[ \sum_j \frac{w_j}{(\gamma_n-\lambda_j)^2} +\int_{-\infty}^\infty
\frac{d\mu(t)}{(\gamma_n-t)^2}\le v_n^{-1} \] for every $n$. We
infer from this that both the sum and the integral on the right-hand
side of \eqref{nonatom} converge absolutely for $z=\gamma_n$.
Indeed, the right-hand side of \eqref{nonatom} vanishes for
$z=\gamma_n$, and so if we put $z=\gamma_n+i\delta$ in
\eqref{nonatom}, divide each side by $iy$, and let $y$ tend to $0$,
we get
\[ v_n^{-1}=\sum_j
\frac{w_j}{(\gamma_n-\lambda_j)^2}+\int_{-\infty}^\infty
\frac{d\mu(t)}{(\gamma_n-t)^2}.\] Since we should have $\|
H_v(\Gamma,\Lambda(\alpha))e^{(n)}\|_{w(\alpha)}=1$, we have reached
a contradiction unless $\mu=0$. On the other hand, if $\mu=0$ and
$c>0$, then we also reach a contradiction because the condition for
orthogonality of the vectors $H_v(\Gamma,\Lambda(\alpha))e^{(n)}$
becomes
\[ \sum_j\left(
\frac{w_j}{\gamma_m-\lambda_j}-\frac{w_j}{\gamma_n-\lambda_j}\right)=0\]
for $m\neq n$, and this is inconsistent with the right-hand side of
\eqref{nonatom} being $0$ whenever $z=\gamma_n$.
\end{proof}

A few remarks are in order. First, it should be noted that we may
have $(\Gamma,v)^*\cap {\Bbb R}=\emptyset$ even if $(\Gamma,v)$ is
an admissible pair. The following is an example. Pick a sequence of
distinct prime numbers $p_l$ such that
\[ \sum_l p_l^{-1/2}<\infty. \]
Set $\Gamma=\bigcup_l p_l^{-1}{\Bbb Z}\setminus {\Bbb Z}$, and equip
$\Gamma$ with the weight sequence $v$ obtained by placing a weight
of magnitude $p_l^{-3/2}$ at every point of the sequence
$p_l^{-1}{\Bbb Z}\setminus {\Bbb Z}$.

On the other hand, if $\Gamma$ is a discrete subset of the real
line, then $H_v(\Gamma,\Lambda(\alpha)): \ell^2_v\to
\ell^2_{w(\alpha)}$ is unitary for every $\alpha$ with one possible
exception: It fails to be unitary when \[  \sum_{n} v_n <\infty \ \
\text{and} \ \ \alpha=\sum_{n} \frac{v_n\gamma_n}{1+\gamma_n^2}.
\] This statement follows almost immediately from
Theorem~\ref{theorem2}. We get the exceptional case because the
constant $c$ in the representation \eqref{nonatom} is obtained as
\[ c=\lim_{y\to \infty} \frac{1}{i y (\alpha-\varphi(i y)) }.\]

If $\Gamma$ is a subset of the unit circle, then the potential
\eqref{phi} should be replaced by \begin{equation} \label{clark}
\varphi(z)= \frac{i}{2}\sum_{n} v_n
\frac{\gamma_n+z}{\gamma_n-z};\end{equation} the analysis goes
through in the same way, and we obtain a statement completely
analogous to Theorem~\ref{theorem2}. Note, however, that for
discrete sets $\Gamma$ on the unit circle, there will be no
exceptional value for $\alpha$ because there is no linear term
`$cz$' in the general representation of a Herglotz function. Indeed,
a Herglotz function $\psi$ in the unit disk is of the form
\[
\psi(z)=b+ \int_{0}^{2\pi} \frac{e^{it}+z}{e^{it}-z} d\mu(t),
\] where $b$ is a real constant and $\mu$
a nonnegative measure on the circle.

Finally, as will be seen in the last section of this paper, the
unitary transformations obtained from Theorem~\ref{theorem2} (and
its counterpart for the unit circle) correspond precisely to Clark's
orthonormal bases \cite{CL}. From this point of view,
Theorem~\ref{theorem2} is essentially a reformulation of Clark's
theorem.

\section{Orthogonal bases of reproducing kernels}

Let $\Hg$ be a Hilbert space of complex-valued functions defined on
some set $\Omega$ in $\CC$. To begin with, we assume that $\Hg$
satisfies the following two axioms:
\begin{itemize}
\item[(A1)]If $f$ is in $\Hg$ and $f(\lambda)=0$ for some point $\lambda$ in
$\Omega$, then we may write $f(z)=f_0(z)(z-\lambda)$ with $f_0$ a
function also belonging to $\Hg$.
\item[(A2)] $\Hg$ has a
reproducing kernel $\kappa_{\lambda}$ at every point $\lambda$ in
$\Omega$.
\end{itemize}
We wish to describe those spaces $\Hg$ which admit orthogonal bases
of reproducing kernels. To avoid trivialities, we assume that the
dimension of $\Hg$ is at least 2. We note that this family of spaces
is part of the much larger family of spaces $\Hg$ that admit Riesz
bases of normalized reproducing kernels. Since each space of the
latter kind can be equipped with an equivalent norm such that one of
the Riesz bases becomes an orthonormal basis, the question of
interest is when a space $\Hg$ has more than one orthogonal basis of
reproducing kernels. It is therefore reasonable to introduce a third
axiom:
\begin{itemize}
\item[(A3)]
There exists a sequence of distinct points $\Gamma=(\gamma_n)$ in
$\Omega$ such that the sequence of normalized reproducing kernels
$\big(\kappa_{\gamma_n}/\|\kappa_{\gamma_n}\|_{\Hg}\big)$
constitutes a Riesz basis for $\Hg$. In addition, there is at least
one point $\lambda$ in $\Omega\setminus \Gamma$ for which
$\kappa_\lambda\neq 0$.
\end{itemize}

The Riesz basis $(\kappa_{\gamma_n}/\|\kappa_{\gamma_n}\|_{\Hg})$
has a biorthogonal basis, which we will call $(g_n)$. Thus
$g_n(\gamma_m)=0$ whenever $m\neq n$. We fix an index $n_0$ and set
$E(z)=g_{n_0}(z)(z-\gamma_{n_0})$. It follows from axiom (A1) that
$E(z)/(z-\gamma_n)=g_{n_0}(z)+(\gamma_n-\gamma_{n_0})g_{n_0}(z)/(z-\gamma_n)$
also belongs to $\Hg$. We use the suggestive notation $E'(\gamma_n)$
for the value of this function at $\gamma_n$. We have
$E'(\gamma_n)\neq 0$ because otherwise $E(z)/(z-\gamma_n)$ would be
identically $0$, which can only happen if all functions in $\Hg$
vanish at every point in $\Omega\setminus \Gamma$; this would
contradict the last part of (A3). By uniqueness of the biorthogonal
sequence $(g_n)$, we now have
\[ g_n(z)=\frac{E(z)}{E'(\gamma_n)(z-\gamma_n)} \]
for every $n$. We call $E$, which is unique up to a multiplicative
constant, the generating function for $\Gamma=(\gamma_n)$. We may
assume that $E$ does not vanish at any point $\lambda$ in
$\Omega\setminus\Gamma$, because then $E(z)/(z-\lambda)$ would be a
vector in $\Hg$ vanishing at every point in $\Gamma$. Hence
$E(z)/(z-\lambda)$ would be identically $0$, which again would be in
contradiction with the second part of (A3).

The sequence $g_n$ is also a Riesz basis for $\Hg$, and therefore
every vector $h$ in $\Hg$ can be written as
\begin{equation} h(z)=\sum_{n}
h(\gamma_n)\frac{E(z)}{E'(\gamma_n)(z-\gamma_n)},
\label{lagrange}\end{equation} where the sum converges with respect
to the norm of $\Hg$ and
\[ \|h\|_{\Hg}^2\simeq \sum_n \frac{|h(\gamma_n)|^2}{\|\kappa_{\gamma_n}\|_{\Hg}^2}
<\infty; \] since point evaluation at every point $z$ is a bounded
linear functional, \eqref{lagrange} also converges pointwise in
$\Omega$. By the assumption that $h\mapsto
(h(\gamma_n)/\|\kappa_{\gamma_n}\|_{\Hg})$ is a bijective map from
$\Hg$ to $\ell^2$, this means that \begin{equation} \sum_n
\frac{\|\kappa_{\gamma_n}\|_{\Hg}^2}{|E'(\gamma_n)|^2|z-\gamma_n|^2}<\infty
\label{pointwise}
\end{equation}
whenever $z$ is in $\Omega\setminus \Gamma$. We see that the
generating function $E$ appears as a common factor in
\eqref{lagrange}. Since $E(z)\neq 0$ for $z$ in
$\Omega\setminus\Gamma$, the function $E$ can be divided out.

We set
\[v_n =\frac{\|\kappa_{\gamma_n}\|_{\Hg}^2}{|E'(\gamma_n)|^2}\]
and observe that since $\Omega\setminus \Gamma$ is assumed to be
nonempty, \eqref{pointwise} implies that
\begin{equation} \sum_n \frac{v_n}{1+|\gamma_n|^2}<\infty.
\label{adm}
\end{equation}
We may now change our viewpoint: Given a sequence of distinct
complex numbers $\Gamma=(\gamma_n)$ and a weight sequence $v=(v_n)$
that satisfy the admissibility condition \eqref{adm}, we introduce
the space $\Hg(\Gamma,v)$ consisting of all functions
\[ f(z)=\sum_n \frac{a_n v_n}{z-\gamma_n} \]
for which
\[ \|f\|_{\Hg(\Gamma,v)}^2=\sum_n |a_n|^2v_n <\infty,\]
assuming that the set $(\Gamma,v)^*$ is nonempty. The reproducing
kernel of $\Hg(\Gamma, w)$ at a point $z$ in $(\Gamma,v)^{*}$ is
\[ k_{z}(\zeta)=\sum_{n}
\frac{v_n}{(\overline{z}-\overline{\gamma_n})(\zeta-\gamma_n)}. \]
If $\Lambda=(\lambda_j)$ is a sequence in $(\Gamma,v)^*$, then we
associate with it the weight sequence $w=(w_j)$, where
\[ w_j=\| k_{\lambda_j}\|_{\Hg(\Gamma,v)}^{-2}=\left(\sum_n
\frac{v_n}{|\lambda_j-\gamma_n|^2}\right)^{-1}. \] Consequently,
$(k_{\lambda_j}/\|k_{\lambda_j}\|_{\Hg(\Gamma,v)})$ is an
orthonormal basis for $\Hg(\Gamma,v)$ if and only if
$H_v(\Gamma,\Lambda):\ell^2_v\to\ell^2_w$ is a unitary
transformation. Thus from the two previous sections we conclude:

If the space $\Hg(\Gamma,v)$ has an orthogonal basis of reproducing
kernels, then $\Gamma$ is a subset of a straight line or a circle.
Moreover, when $\Gamma$ is a subset of the real line, the orthogonal
bases of reproducing kernels for $\Hg(\Gamma,w)$ are obtained from
the unitary transformations described by Theorem~\ref{theorem2}; an
analogous result holds when $\Gamma$ is a subset of the unit circle.

\section{Relation to Clark's bases}

We are now finally prepared to point out the correspondence between
our description of unitary discrete Hilbert transforms and the bases
studied by de Branges \cite{DB} and Clark \cite{CL}. We restrict to
Clark's bases; the only difference between the two cases is that
Clark considered the case of the unit circle while de Branges worked
on the real line with, in our terminology, $|\gamma_n|\to \infty$.

Suppose $\varphi$ is of the form \eqref{clark} with
$\Gamma=(\gamma_n)$ a sequence of distinct points on the unit
circle. Then the function
\[ I(z)=\frac{\varphi(z)-i}{\varphi(z)+i} \]
is an inner function in the open unit disk ${\Bbb D}$. We associate
with $I$ the so-called model space $K_I^2=H^2\ominus I H^2$, which
is the orthogonal complement to the shift-invariant subspace $I H^2$
of the Hardy space $H^2$ of the unit disk. Since
$1/(1-\overline{\zeta}z)$ is the reproducing kernel for $H^2$ at a
point $\zeta$ in ${\Bbb D}$, the reproducing kernel for $K^2_I$ at
the same point $\zeta$ is
\[ \kappa_{\zeta}(z)=\frac{1-\overline{I(\zeta)}
I(z)}{1-\overline{\zeta} z}. \] This formula extends to each point
on the unit circle at which every function in $K^2_I$ has a radial
limit whose modulus is bounded by a constant times the $H^2$ norm of
the function.

A computation shows that
\[ i \frac{1+I(z)}{1-I(z)}=\varphi(z) \]
which according to Clark's theorem means that the reproducing
kernels
\[ \kappa_{\gamma_n}(z)=\frac{1-I(z)}{1-\overline{\gamma_n}z}\]
constitute an orthogonal basis for $K^2_I$. In fact, Clark's theorem
says that if $\beta$ is a point on the unit circle and the spectral
measure of the Herglotz function
\[  \varphi_\beta(z)=i \frac{\beta+I(z)}{\beta-I(z)} \]
is purely atomic, then the reproducing kernels associated with the
spectrum of $\varphi_\beta$ also constitute an orthogonal basis for
$K^2_I$. The spectral measures generated in this way correspond
precisely to the spectral measures of the functions
$1/(\alpha-\varphi(z))$ with $\alpha$ any real number.

Having observed this correspondence, we conclude that a Hilbert
space $\Hg$ of the type considered in the previous section can have
more than one orthogonal basis of reproducing kernels only if $\Hg$
is, up to trivial modifications, a model space $K^2_I$ either in the
unit disk or in the upper half-plane.\footnote{It may be objected
that we fall short of characterizing all Clark bases because there
exist model subspaces that do not satisfy axiom (A1) above. However,
axiom (A1) can be relaxed in an appropriate manner so that this
shortcoming is removed; we refer to our recent paper \cite{BMS} for
details.}

An additional wonder, which can be seen from Clark's theorem or
indeed by a straightforward computation, is that the norm in $\Hg$
can always be computed as an $L^2$ integral over a circle or a
straight line.


\begin{thebibliography}{BRSHZE}

\bibitem{BMS} Y. Belov, Tesfa Y. Mengestie, and K. Seip,
\emph{Discrete Hilbert transforms on sparse sequences},
arXiv:0912.2899v1, 2009.
\bibitem{CY}  J. Cima, A. Matheson, and W. Ross,
\emph{The Cauchy Transform}, Math. Surveys Monogr. \textbf{125},
Amer. Math. Soc., Providence, RI, 2006.
\bibitem{CL}  D. N. Clark, \emph{One dimensional perturbations of restricted
shifts,} J. Analyse Math. \textbf{25} (1972), 169--191.
\bibitem{DB} L.  de Branges, \emph{Hilbert Spaces of Entire
Functions,} Prentice--Hall, Englewood Cliffs, 1968.




\end{thebibliography}
\end{document}